\documentclass[12pt,reqno]{article}

\usepackage[paper=letterpaper,margin=1in,twoside=false]{geometry}

\usepackage[T3,T1]{fontenc}

\usepackage{amsmath,amssymb,amsthm} 
\usepackage{enumerate,paralist}
\usepackage{graphicx}
\usepackage{mathtools}

\usepackage{todonotes}

\usepackage{ifthen}

\usepackage[colorlinks,urlcolor=blue]{hyperref}

\newtheorem{theorem}{Theorem}
\newtheorem{lemma}[theorem]{Lemma} 
\newtheorem{proposition}[theorem]{Proposition} 
\newtheorem{corollary}[theorem]{Corollary}

{\theoremstyle{definition} 

\newtheorem{conjecture}[theorem]{Conjecture}

 }

\newtheoremstyle{named}%
  {}{}						
  {\upshape}				
  {0pt}{\bfseries}			
  {.}						
  {.5em}					
  {\thmname{#1}\thmnote{ #3}}  

\theoremstyle{named}

\newcommand{\of}{\subseteq}

\newcommand{\symd}{\bigtriangleup}
\newcommand{\adj}{\sim}
\newcommand{\nadj}{\not\adj}

\newcommand{\divides}{\mid}

\newcommand{\N}{\mathbb{N}}

\newcommand{\A}{\mathcal{A}}
\newcommand{\C}{\mathcal{C}}

\renewcommand{\d}{\delta}
\newcommand{\e}{\epsilon}

\newcommand\fall[1][t]{_{(#1)}}

\newcommand{\Nof}[2]{N_{#1}({#2})}
\renewcommand{\NG}[1]{\Nof{#1}{G}}
\newcommand{\NTG}{\NG{T}}

\DeclareMathOperator{\ex}{ex}

\DeclareMathOperator{\mex}{mex}

\DeclareMathOperator{\CT}{CT}

\let\k\relax
\DeclareMathOperator{\k}{k}

\DeclarePairedDelimiter{\abs}{|}{|}
\DeclarePairedDelimiter{\set}{\{}{\}}
\DeclarePairedDelimiterX\setof[2]{\{}{\}}{#1\,:\,#2}
\DeclarePairedDelimiter{\parens}{(}{)}

\begin{document}

\title{Stability and Erd\H{o}s--Stone type results for $F$-free graphs with a fixed number of edges}

\date{\today}
\author{Jamie Radcliffe\\\small{University of Nebraska-Lincoln}\\\small{\texttt{jamie.radcliffe@unl.edu}} 
\and Andrew Uzzell\\\small{Grinnell College}\\\small{uzzellan@grinnell.edu}}
\maketitle
\begin{abstract}
    A fundamental problem of extremal graph theory is to ask, ``What is the maximum number of edges in an $F$-free graph on $n$ vertices?'' Recently Alon and Shikhelman proposed a more general, subgraph counting, version of this question. They considered the question of determining the maximum number of copies of a fixed graph $T$ in an $F$-free graph on $n$ vertices. 
    
    In this more general context, where we are no longer counting edges, it is also natural to ask what is the maximum number of copies of $T$ in an $F$-free graph with $m$ edges and no restriction on the number of vertices. Frohmader, in a different context, determined the answer when $T$ and $F$ are both complete graphs. We prove results for this problem analogous to the Erd\H{o}s--Stone theorem, the Erd\H{o}s--Simonovits theorem, and the stability theorem of Erd\H{o}s--Simonovits.
\end{abstract}

\maketitle

\section{Introduction} 
\label{sec:introduction}

\subsection{Extremal numbers and generalizations of extremal numbers} 
\label{sub:extremal_numbers_and_generalizations}

The fundamental problem of extremal graph theory is to compute the extremal number,
\[
    \ex(n,F) = \max\setof{e(G)}{\text{$G$ is an $F$-free graph on $n$ vertices}}.
\]
Recently Alon and Shikhelman~\cite{AlonShikhelman16} proposed a more general version of this problem. Rather than counting edges, they considered the  problem of determining the maximum number of copies of a fixed graph $T$ in an $F$-free graph on $n$ vertices. Letting $\NTG$ be the number of copies of $T$ in $G$, we define
\[
    \ex_T(n,F) = \max\setof{\NTG}{\text{$G$ is an $F$-free graph on $n$ vertices}}.
\] 

Tur\'an's theorem \cite{Turan41} states that $\ex(n,K_{r+1}) = t_{r}(n)$, the number of edges in the Tur\'an graph $T_r(n)$, the complete $r$-partite graph on $n$ vertices with parts as equal in size as possible. Moreover the Tur\'an graph is the unique extremal graph. The following result was proved by Zykov~\cite{Zykov49} (and has since been rediscovered many times).

\begin{theorem}\label{th:Zykov}
    For all $n \geq r \geq t \geq 2$, the maximum in the definition of $\ex_{K_t}(n, K_{r+1})$ is uniquely achieved by the Tur\'an graph~$T_r(n)$. In other words,
    \[
        \ex_{K_t}(n,K_{r+1}) = \Nof{K_t}{T_r(n)}.
    \]
\end{theorem}

Now that we are counting copies of $T$, rather than edges, it makes sense to shift away from our resource being a limited number of vertices we are allowed, and consider similar problems for the class of graphs with $m$ edges. We make a third parallel definition.
\[
    \mex_T(m,F) = \max\setof{\NTG}{\text{$G$ is an $F$-free graph with $m$ edges}}.
\]
It is important to note that this definition does not place any restriction on the number of vertices of~$G$.


\subsection{Previous results about $\mex_T(m,F)$} 
\label{sub:previous_results}

Some results about $\mex_T(m,F)$ are known (though not using that terminology). One can even think of the Kruskal--Katona theorem~\cite{Katona68,Kruskal63} as proving a result in this direction. We start with a little background about that theorem.

Given $n$,~$k \in \N$, let $\binom{[n]}{k}$ denote the family of $k$-sets of~$[n]$.  The \emph{colexicographic} or \emph{colex order} on $\binom{[n]}{k}$ is defined as follows: for all $A$,~$B \in \binom{[n]}{k}$, $A < B$ if and only if $\max(A \symd B) \in B$.  For a family $\A\of \binom{[n]}k$ we define the \emph{shadow} of $\A$ on level $p<k$ to be the set
\[
    \partial_p(\A) = \setof[\Big]{B\in \binom{[n]}p}{\exists A\in \A \text{ s.t.\ } B \subseteq A}.
\]
The Kruskal--Katona theorem gives a bound for the minimum size of $\partial_p(\A)$ as a function of the size of $\A$.

\begin{theorem}\label{thm:KK}
    If $\A\of \binom{[n]}k$ and $\C$ is the colex initial segment of $\binom{[n]}k$ of size $\abs{\A}$ then for any $p<k$ we have 
    \[
        \abs{\partial_p(\A)} \ge \abs{\partial_p(\C)}.
    \]
    One should also note $\partial_p(\C)$ is itself an initial segment in the colex order on $\binom{[n]}{p}$. 
\end{theorem} 

It is an immediate corollary of Theorem~\ref{thm:KK} that for every $m$ and~$t$, the maximum number of copies of~$K_t$ in a graph with $m$ edges is achieved by the graph with vertex set~$[n]$ whose edge set consists of the first $m$ elements of~$\binom{[n]}{2}$ in colex order. We call this graph the \emph{colex graph} with $m$ edges, and denote it $C(m)$. This is a slight abuse of notation, since we have not specified $n$, but in our problems we only care that we have enough vertices, not how many there actually are.

Frohmader~\cite{Frohmader08} determined the value of $\mex_{K_s}(m, K_{r+1})$ for all $r \geq s \geq 3$.  His results were phrased in terms of simplicial complexes, so let us take a moment to recall the relevant definitions.  

Let $\Delta$ be a simplicial complex.  If $F$ is a face of~$\Delta$, then the \emph{dimension} of~$F$ is $\dim F = |F| - 1$.  The \emph{dimension} of~$\Delta$ is~$\dim \Delta = \max_{F \in \Delta} \dim F$.  Let $d = \dim \Delta + 1$ and, for each $i$, $-1 \leq i \leq d-1$, let $f_i$ denote the number of $i$-dimensional faces in $\Delta$.  Recall that the \emph{$f$-vector} of~$\Delta$ is the $d$-tuple~$(f_0, \dots, f_{d-1})$.  We say that a complex~$\A$ is \emph{$r$-colorable} if there is a partion of its vertex set into $r$ parts such that each set in $\A$ meets each part in at most one element.

A simplicial complex~$\Delta$ is called a \emph{flag complex} if every minimal non-face of~$\Delta$ has two elements.  This is equivalent to the notion of a ``clique complex'': the \emph{clique complex} of a graph~$G$ is the simplicial complex~$\Gamma$ whose vertex set is~$V(G)$ and whose faces are the cliques of~$G$.  It is easy to see that a flag complex is $r$-colorable if and only if it is the clique complex of an $r$-colorable graph.  We say that a complex~$\Delta$ is \emph{balanced} if $\dim \Delta = d-1$ and $\Delta$ is $d$-colorable.

Kalai (unpublished; see~\cite[p.~100]{Stanley96}) and Eckhoff~\cite{Eckhoff88} conjectured that if $\Delta$ is a flag complex, then there exists a balanced  complex~$\Gamma$ with the same $f$-vector as~$\Delta$.  Frohmader~\cite{Frohmader08} proved the Kalai--Eckhoff conjecture. This is in fact sufficient to determine $\mex_{K_t}(m,K_{r+1})$. For completeness we include a proof of this deduction below. 

We will need to quote the ``colored'' version of the Kruskal--Katona theorem, proved by Frankl, F\"uredi, and Kalai~\cite{FranklFurediKalai88}. The role played by the colex order in the Kruskal--Katona theorem is played here by the \emph{$r$-partite colex order}. This is colex order restricted to subsets $A$ of $\N$ such that for all $i, j\in A$ we have $i=j$ or $i \not\equiv j \pmod{r}$.

Given $m$ and $r$, the \emph{colex Tur\'an graph}~$CT_r(m)$ is the graph on vertex set $\N$ whose edge set consists of the first $m$ edges in $r$-partite colex order. (See Figure~\ref{fig:CT}.) Note that if $m = t_r(n)$, then the unique non-trivial component of $CT_r(m)$ is isomorphic to  $T_r(n)$.

\begin{theorem}[\cite{FranklFurediKalai88}]\label{th:FFK}
    If $\A\of \binom{[n]}k$ is $r$-colorable  and $\C$ is the initial segment of $\binom{[n]}k$ in the $r$-partite colex order of size $\abs{\A}$, then for any $p<k$ we have 
    \[
        \abs{\partial_p(\A)} \ge \abs{\partial_p(\C)}.
    \]
    One should also note $\partial_p(\C)$ is itself an initial segment in the $r$-partite colex order on $\binom{[n]}{p}$. 
\end{theorem} 

In the next corollary and throughout the rest of the paper we write $\k_t(G)$ for $N_{K_t}(G)$.  Also, given $v \in V(G)$ and $e \in E(G)$, let $\k_s(v)$ and $\k_s(e)$ denote the number of copies of~$K_s$ in $G$ that contain $v$ and the number of copies of~$K_s$ in $G$ that contain $e$, respectively.
 
\begin{corollary}\label{cor:FFK}
    If $G$ is an $r$-partite graph with $m$ edges then 
    \[
        \k_t(G) \le \k_t(\CT_r(m)).
    \]
\end{corollary}
\begin{proof}
    This is an immediate consequence of Theorem~\ref{th:FFK} and the definition of $\CT_r(m)$.
\end{proof}
\begin{figure}[ht]
	\begin{center}
		\begin{tikzpicture}[scale=0.7]
			\foreach \i in {0,1,2}
			{
				\foreach \j [evaluate=\j as \n using {int(3*\j+\i+1)}] in {0,1,2,3}
				{
					\path (90-\i*120:1.2*\j+2) node[draw, circle, inner sep=0pt, minimum size=0.6cm] (\n) {$\n$};
				}
			}
			\foreach \i in {0,1,2}
				\draw[very thick, loosely dotted] (90-\i*120:1.2*3.5+2) -- (90-\i*120:1.2*4.5+2);
			\foreach \n in {1,...,8}
				\foreach \m in {1,...,7}
				{
					\pgfmathparse{int(mod(\n,3)-mod(\m,3))} \let\z\pgfmathresult
					\ifthenelse{\z=0}{}{\draw (\n) -- (\m);}
				}
			\def\n{9}
			\foreach \m in {1,...,6}
			{
				\pgfmathparse{int(mod(\n,3)-mod(\m,3))} \let\z\pgfmathresult
				\ifthenelse{\z=0}{}{\draw (\n) -- (\m);}
			}
		\end{tikzpicture}
	\end{center}
	\caption{The graph $CT_3(25)$. It contains the $21=t_3(8)$ edges of $T_3(8)$, as well as the edges $19$, $29$, $49$, and $59$.}
    \label{fig:CT}
\end{figure}
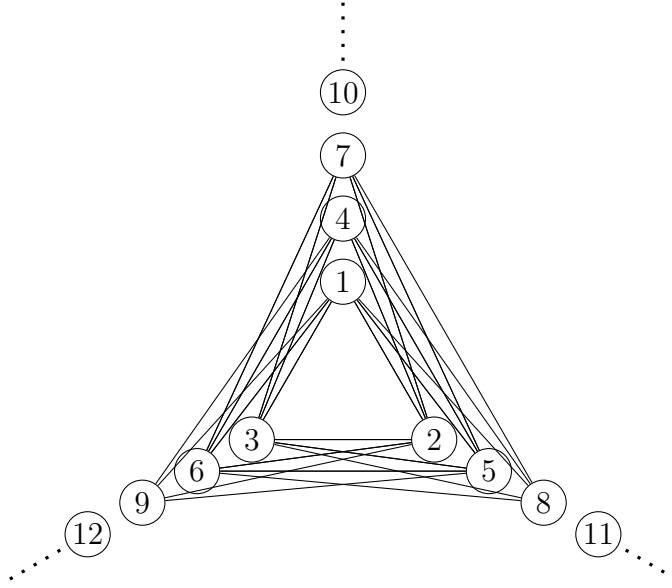

From these results we can prove that $\CT_r(m)$ achieves the maximum in the definition of $\mex_{K_t}(m, K_{r+1})$.

\begin{theorem}[Frohmader~\cite{Frohmader08}]\label{th:Froh}
    For all $r\ge t\ge 2$ we have
    \[
        \mex_{K_t}(m, K_{r+1}) = \k_t(\CT_r(m)).
    \]
\end{theorem}

\begin{proof}
    Consider a $K_{r+1}$-free graph $G$ having $m$ edges. Its clique complex $K$ is a flag complex, so, since the Kalai--Eckhoff conjecture is true, there is a balanced complex $\Gamma$ having the same $f$-vector as~$K$. Since $\dim \Gamma = \dim K \le r-1$ we know that $\Gamma$ is $r$-colorable. By Corollary~\ref{cor:FFK} we have $\k_t(G)\le \k_t(\CT_r(m))$.
\end{proof}


\subsection{Structural supersaturation and stability} 
\label{sub:structural_supersaturation_supersaturation_and_stability}

Tur\'an's theorem has inspired a great deal of research on the size and structure of extremal and near-extremal $F$-free graphs.  We mention several important theorems in this area in order to motivate our results.

Given a graph~$G$ and a positive integer~$t$, we let $G[t]$ denote the $t$-fold blowup of~$G$ (in which every vertex of $G$ is replaced by an independent set of size~$t$ and every edge by a copy of $K_{t, t}$). Erd\H{o}s and Stone~\cite{ErdosStone46} showed that a graph with $n$ vertices and $t_r(n) + \e n^2$ edges not only contains  $K_{r+1}$, but contains a sizable blow-up of~$K_{r+1}$.

\begin{theorem}\label{th:ErdosStone}
For all $r \geq 2$, $t \geq 1$, and $\e > 0$, there exists $n_0$ such that if $n \geq n_0$ and $G$ is a graph with $n$ vertices such that
	\[
		e(G) \ge t_r(n) + \e n^2,
	\]
then $G$ contains $K_{r+1}[t]$.
\end{theorem}

Erd\H{o}s and Simonovits~\cite{ErdosSimonovits66} observed that Theorem~\ref{th:ErdosStone} implies that $\chi(F)$ determines $\ex(n, F)$ up to a $o(n^2)$ error term.

\begin{theorem}\label{th:ErdosSimonovits}
Let $F$ be a graph.  We have
	\[
		\ex(n, F) = \parens[\bigg]{\frac{\chi(F)-2}{\chi(F)-1}+o(1)} \binom{n}{2}.
	\]
\end{theorem}

Erd\H{o}s and Simonovits~\cite{Erdos67,Simonovits68} also proved a stability result, which says that a $K_{r+1}$-free graph with nearly the extremal number of edges has nearly extremal structure.

\begin{theorem}\label{th:stability}
For all $r \geq 2$ and all $\e > 0$, there exist $n_0$ and $\d > 0$ such that if $n \geq n_0$ and $G$ is a $K_{r+1}$-free graph with $n$ vertices such that
	\[
		e(G) \geq (1 - \d)t_r(n),
	\]
then $G$ can be made $r$-partite by deleting at most~$\epsilon n^2$ edges.
\end{theorem}

\subsection{Results} 
\label{sub:results}

In Section~\ref{sec:SSS} we prove analogues of the Erd\H{o}s--Stone theorem (Theorem~\ref{th:ErdosStone}), the Erd\H{o}s--Simonovits theorem (Theorem~\ref{th:ErdosSimonovits}), and the Erd\H{o}s--Simonovits stability theorem (Theorem~\ref{th:stability}) in the context of Frohmader's theorem, Theorem~\ref{th:Froh}. To be precise we prove the following results.

\begin{theorem}\label{th:mexErdosStone}
For all $r, s \geq 3$, all $t \geq 1$, and all $\e > 0$, there exists $m_0$ such that if $m \geq m_0$ and $G$ is a graph with $m$ edges such that
	\[
		\k_s(G) \ge \mex_{K_s}(m,K_{r+1}) + \e m^{s/2},
	\]
then $G$ contains $K_{r+1}[t]$.
\end{theorem}

\begin{theorem}\label{th:mexErdosSimonovits}
Let $r, s \geq 3$, let $F$ be a graph, and let $\chi(F)=r+1$.  We have
	\[
		\mex_{K_s}(m,F) = \mex_{K_s}(m,K_{r+1}) + o(m^{s/2}).
	\]
\end{theorem}

Theorem~\ref{th:mexErdosSimonovits} follows from Theorem~\ref{th:mexErdosStone} in much the same way that Theorem~\ref{th:ErdosSimonovits} follows from Theorem~\ref{th:ErdosStone}, so we will omit the proof.

\begin{theorem}\label{th:mexStability}
For all $r \geq s \geq 3$ and all $\e > 0$, there exist $m_0$ and $\d > 0$ such that if $m \geq m_0$ and $G$ is a $K_{r+1}$-free graph with $m$ edges such that
	\[
		\k_s(G) \geq (1 - \d)\mex_{K_s}(m, K_{r+1}),
	\]
then $G$ can be made $r$-partite by deleting at most~$\epsilon m$ edges.
\end{theorem}

\begin{corollary}\label{cor:mexStabilityGeneral}
For all $r \geq s \geq 3$, every graph~$F$ with chromatic number~$r+1$, and all $\e > 0$, there exist $m_0$ and $\d > 0$ such that if $m \geq m_0$ and $G$ is an $F$-free graph with $m$ edges such that
	\[
		\k_s(G) \geq (1 - \d)\mex_{K_s}(m, F),
	\]
then $G$ can be made $r$-partite by deleting at most~$\epsilon m$ edges.
\end{corollary}

Our results establish a number of parallels between $\mex_T(n,F)$ and~$\ex(n,F)$, with the colex Tur\'an graph~$CT_r(m)$ playing the role in results about $\mex_T(n,F)$ that the Tur\'an graph~$T_r(n)$ plays in results about $\ex(n,F)$.  However, this correspondence is not perfect.  

Let $F$ be a graph with a critical edge.  Simonovits~\cite{Simonovits68} used the stability method to determine $\ex(n, F)$ (and the extremal graph) for all $n$~sufficiently large.

\begin{theorem}\label{th:criticalEdge}
Let $F$ be a graph with $\chi(F) = r + 1$ and suppose that $F$ contains an edge~$e$ such that $\chi(F - e) = r$.  For all $n$~sufficiently large, $\ex(n, F) = t_r(n)$ and $T_r(n)$ is the unique extremal graph.
\end{theorem}

In contrast, if $F$ is as in the statement of Theorem~\ref{th:criticalEdge} and $\d(F) > r$, there exist infinitely many values of~$m$ such that $CT_r(m)$ is not an extremal graph for $\mex_{K_s}(m,F)$.

Given $m$, let $n$ be the least integer such that $m \leq t_r(n)$.  Let $G$ be the graph consisting of~$T_r(n-1)$ and a vertex~$v^*$ that is joined to $m - t_r(n-1)$ vertices of~$T_r(n-1)$, distributed as evenly as possible among the $r$ classes of~$T_r(n-1)$.  Observe that if $r \leq m - t_r(n-1) < \d(F)$, then $G$ is $F$-free but not $r$-partite.  Moreover,
\begin{align*}
\k_s(G) &= \k_s(T_r(n-1)) + \k_s(v^*) \\
&= \k_s(T_r(n-1)) + \k_{s-1}\bigl(T_r\bigl(m - t_r(n-1)\bigr)\bigr) \\
&> \k_s(T_r(n-1)) + \k_{s-1}\bigl(T_{r-1}\bigl(m - t_r(n-1)\bigr)\bigr)\\
&= \k_s(CT_r(m)).
\end{align*}

Finally there are a number of very natural analogues of results concerning $\ex(n,F)$ that are open for $\mex_T(m,F)$. In Section~\ref{sec:open} we briefly discuss some of these open problems.


\section{Proof of Theorems~\ref{th:mexErdosStone}, \ref{th:mexErdosSimonovits}, and \ref{th:mexStability}} 
\label{sec:SSS}

\subsection{Preliminaries and notation} 
\label{sub:preliminaries_and_notation}

Let $G$ be a graph and let $s \geq 2$.  Recall that for $v \in V(G)$ and $e \in E(G)$, $\k_s(v)$ and $\k_s(e)$ denote the number of copies of~$K_s$ in $G$ that contain $v$ and the number of copies of~$K_s$ in $G$ that contain $e$, respectively.  The minimum values of these quantities are denoted
\begin{align*}
    \d_s(G) &= \min\setof{\k_s(v)}{v \in V(G)} \\
    \d'_s(G) &= \min\setof{\k_s(e)}{e \in E(G)}.
\end{align*}

In the extremal $K_{r+1}$-free graph $\CT_r(m)$ the average degree is a multiple of $m^{1/2}$ and the number of copies of $K_s$ is a multiple of $m^{s/2}$. We define those constant multiples here. Given $r \geq 2$ and $s \geq 3$, let
\begin{equation}\label{eq:beta}
\beta_r = \biggl(\frac{2(r-1)}{r}\biggr)^{1/2}
\end{equation}
and let
\begin{equation}\label{eq:c_s}
c_{r, s} = \dfrac{\binom{r}{s}}{\binom{r}{2}^{s/2}}.
\end{equation}

The following simple proposition collects some computations about $\k_s(\CT_r(m))$.

\begin{proposition}\label{prop:colexTuranCliques}
    If $r \divides n$ and $m = t_r(n)$, then $m = \bigl(\frac{n}{r}\bigr)^2\binom{r}{2}$ and so 
    \[
        \k_s(\CT_r(m)) = \k_s(T_r(n)) = \biggl(\frac{n}{r}\biggr)^s\binom{r}{s} = c_{r, s} m^{s/2}.
    \]
    In particular in this case $\CT_r(m)$ is $\beta_r m^{1/2}$-regular. Furthermore we have
    \begin{equation*}\label{eq:asymptoticKs}
    \k_s(\CT_r(m)) = c_{r, s} m^{s/2} + O(m^{(s-1)/2}).
    \end{equation*}
\end{proposition}
\begin{proof}
    Straightforward.
\end{proof}

We also record some properties of the constants $\beta_r$ and $c_{r, s}$ defined above.

\begin{proposition}\label{prop:beta c_s}
For all $r \geq 2$ and $s \geq 3$,
\begin{align}%
\dfrac{\binom{r-1}{s-1}}{(r-1)^{s-1}}\beta_r^{s-2} &= \frac{s}{2}\,c_{r, s} \label{eq:beta c_{r, s}} \\
\intertext{and}
    c_{r, s} &= \frac{2^{s/2}}{s!} \cdot \frac{r_{(s)}}{r^{s/2}(r-1)^{s/2}} \leq \frac{2^{s/2}(r-2)}{s!(r-1)}. \label{eq:c_sUB}
\end{align}
\end{proposition}

\begin{proof}
If $r < s$, then~\eqref{eq:beta c_{r, s}} holds trivially, as both sides equal~$0$.  If $r \geq s \geq 3$, \eqref{eq:beta} and~\eqref{eq:c_s} imply that
\begin{align*}
\dfrac{\binom{r-1}{s-1}}{(r-1)^{s-1}}\beta_r^{s-2} &= \dfrac{\binom{r-1}{s-1}}{(r-1)^{s-1}}\biggl(\dfrac{2(r-1)}{r} \biggr)^{\frac{s-2}{2}}\\
&= \dfrac{r}{2} \cdot \dfrac{(r-1)\fall[s-1]}{(s-1)!}\cdot \biggl(\dfrac{2}{r(r-1)} \biggr)^{s/2}\\
&=\dfrac{s}{2} \cdot \dfrac{r\fall[s]}{s!}\cdot \binom{r}{2}^{-s/2}\\
&=\dfrac{s}{2}c_{r, s}.
\end{align*}

Also, \eqref{eq:c_sUB} is immediate from~\eqref{eq:c_s}.
\end{proof}

We will need a consequence of the Kruskal--Katona theorem noted by Lov\'asz~\cite[Exercise~13.31]{Lovasz79}.

\begin{theorem}\label{th:LovaszKK}
Let $s \geq 3$ and let $x \geq 0$.  If $G$ is a graph with $\binom{x}{2}$ edges, then $\k_s(G) \leq \binom{x}{s}$.
\end{theorem}

\begin{corollary}\label{cor:LovaszKKcor}
Let $s \geq 3$.  If $G$ is a graph with $m$ edges, then 
\[
    \k_s(G) \leq \bigl(1+o(1)\bigr)\frac{2^{s/2}}{s!}m^{s/2}.
\]
\end{corollary}
\begin{proof}
    Straightforward.
\end{proof}


We will also need the following result of Erd\H{o}s, Frankl, and R\"odl~\cite{ErdosFranklRodl86}. 

\begin{theorem}\label{th:removal}
Let $r \geq 2$.  For all $\eta > 0$ and every graph~$F$ with chromatic number~$r+1$, there exists $n_0$ such that if $G$ is an $F$-free graph of order~$n \geq n_0$, then $G$ can be made $K_{r+1}$-free by removing at most~$\eta n^2$ edges.
\end{theorem}


\subsection{Proof of Theorem~\ref{th:mexErdosStone}} 
\label{sub:proof_of_theorem_ref_th_mexerdosstone}

\begin{proof}[Proof of Theorem~\ref{th:mexErdosStone}]
First, we show that if $m$ is sufficiently large, then $G$ contains a subgraph~$G'$ that has both positive edge density and many copies of~$K_s$ relative to~$e(G')$.

Let $\rho$ be such that
\begin{equation}\label{eq:rho}
1 > \rho > 1 - \biggl(\frac{s!}{2^{s/2}}\Bigl(c_{r, s}+\frac{\e}{3}\Bigr)\biggr)^{2/s}.
\end{equation}
(Proposition~\ref{prop:beta c_s} implies that for all $r$ and $s$, if $\e$ is sufficiently small, then the right-hand side of~\eqref{eq:rho} is positive.)

Let $m$ be sufficiently large.  If $\d'_s(G) \geq \frac{2^{s-2}\e^{2s-4}}{(s-2)!} m^{(s-2)/2}$, we do nothing.  Otherwise, we let $G_0 = G$ and, for each $i \geq 0$, if $G_i$ contains an edge~$e_i$ with $\k_s(e_i) <  \frac{2^{s-2}\e^{2s-4}}{(s-2)!} (e(G_i))^{(s-2)/2}$, we set $G_{i+1} = G_i - e_i$. 

Suppose that we delete $\lfloor \rho m \rfloor $ such edges and let $G' = G_{\lfloor \rho m \rfloor}$ denote the resulting subgraph.  We have
\begin{align*}
\k_s(G') &= \k_s(G) - \sum_{i=0}^{\lfloor \rho m \rfloor - 1} \k_s(e_i) \\
&\geq \k_s(G) - \sum_{i=0}^{\lfloor \rho m \rfloor - 1}  \frac{2^{s-2}\e^{2s-4}}{(s-2)!} (m-i)^{(s-2)/2} \\
& \geq \k_s(G) - \frac{2^{s-2}\e^{2s-4}}{(s-2)!} (1+\e^2)\frac{2}{s} \bigl(m^{s/2} - \bigl((1-\rho) m\bigr)^{s/2}  \bigr).
\end{align*}
Thus, using~\eqref{eq:rho} twice, we have
\begin{align*}
\k_s(G') &\geq \k_s(G) - \frac{2^{s-2}\e^{2s-4}}{(s-2)!} \rho m^{s/2} \\
&> (c_{r, s} + \e - 2\e^{2s-4})m^{s/2} \\
&> \Bigl(c_{r, s} + \frac{2\e}{3}\Bigr)m^{s/2} \\
&> \biggl(\frac{2^{s/2}}{s!}(1 - \rho)^{s/2}+\frac{\e}{3}\biggr) m^{s/2},
\end{align*}
which contradicts Corollary~\ref{cor:LovaszKKcor}.

So, $G$ has a subgraph~$G'$ with $m' > (1-\rho)m$ edges and $n'$ vertices such that 
\[
    \d'_s(G') \geq \frac{2^{s-2}\e^{2s-4}}{(s-2)!} (m')^{(s-2)/2}.
\]

We claim that
\begin{equation}\label{eq:cliquesInG'}
\k_s(G') \geq (c_{r, s} + \e) (m')^{s/2}.
\end{equation}
Indeed, given $i \geq 1$, suppose that $\k_s(G_{i-1}) \geq (c_{r, s} + \e) e(G_{i-1})^{s/2}$.  If $\e$ is sufficiently small, then we have
\begin{align*}
\k_s(G_i) &\geq \k_s(G_{i-1}) - \frac{2^{s-2}\e^{2s-4}}{(s-2)!} e(G_{i-1})^{(s-2)/2} \\
&> (c_{r, s} + \e) e(G_{i-1})^{s/2} - (c_{r, s} + \e) \frac{s}{4}  e(G_{i-1})^{(s-2)/2} \\
&> (c_{r, s} + \e) (e(G_{i-1})-1)^{s/2} \\
&= (c_{r, s} + \e) e(G_{i})^{s/2}.
\end{align*}
The claimed inequality~\eqref{eq:cliquesInG'} follows by induction on $i$ and our assumption on $G = G_0$.

Observe that if $e \in E(G')$ and $v$ is an endpoint of~$e$, then $\k_s(e) \leq \k_{s-1}(v) \leq \binom{d(v)}{s-2}$.  It follows that
\[
\frac{2^{s-2}\e^{2s-4}}{(s-2)!} (m')^{(s-2)/2} \leq \d'_s(G') \leq \binom{\d(G')}{s-2} \leq \binom{2m'/n'}{s-2} \leq \frac{(2m'/n')^{s-2}}{(s-2)!},
\]
whence
\begin{equation}\label{eq:verticesLB}
n' \leq \frac{1}{\e^2} (m')^{1/2}.
\end{equation}

Suppose that $G'$ does not contain a copy of~$K_{r+1}[t]$.  By the trivial bound $n' > \sqrt{2m'}$, we may let $n'$ be as large as we wish by taking $m'$ to be sufficiently large.  So, by~\eqref{eq:verticesLB} and Theorem~\ref{th:removal}, if $m'$ is sufficiently large, then we can delete all copies of~$K_{r+1}$ in $G'$ by removing at most~$\e^{4s-3} (n')^2 \leq \e^{4s-7} m'$ edges.  This means that we remove at most~$\e^{4s-3} (n')^s \leq \e^{2s-3} (m')^{s/2}$ copies of~$K_s$.  Let $G''$ denote the resulting graph and let $m'' = e(G'')$.  By~\eqref{eq:cliquesInG'}, Proposition~\ref{prop:colexTuranCliques}, and Theorem~\ref{th:Froh}, we have
\begin{align*}
    \k_s(G'') \geq \k_s(G') - \e^{2s-3} (m')^{s/2}  &\geq (c_{r, s} + \e - \e^{2s-3})(m')^{s/2} \\
            &\geq \parens[\Big]{c_{r, s}(1 - \e^{4s-7})^{s/2} + \frac{\e}{2}} (m')^{s/2} > \mex_{K_s}(m'', K_{r+1}),
\end{align*}
a contradiction.
\end{proof}


\subsection{Proof of Theorem \ref{th:mexStability}} 
\label{sub:proof_of_theorem_ref_th_mexstability}

Proofs of stability results in extremal graph theory often begin by showing that a global density assumption on a graph~$G$ implies a minimum degree condition.  This is frequently accomplished by iteratively deleting vertices of degree at most~$\alpha |V(G)|$ (where $\alpha > 0$ is an appropriate constant) and showing that the density of~$G$ and the forbidden subgraph condition mean that only a small fraction of the vertices could have been deleted in this way.

However, in our case, if we delete vertices whose degree is too small as a function of the number of vertices of~$G$, then there is no reason to expect that the process will terminate quickly, for the simple reason that we do not know how many vertices $G$ has.  In particular, we may end up deleting far more than~$\e m$ edges.  Instead, letting $S$ denote the set of the vertices of~$G$ whose degree is too small as a function of the number of \emph{edges} of~$G$, we will show that the vertices of~$S$ span only a small fraction of the edges of~$G$.  We will then be able to show that $G - S$ has high minimum degree as a function of the number of vertices.

\begin{lemma}\label{le:denseSubgraph}
Given $r\ge s \geq 3$ and  $\e > 0$,  there exist $m_0$ and $\d > 0$ with the following property. If $m \geq m_0$ and $G$ is a $K_{r+1}$-free graph with $m$ edges such that
	\[
		\k_s(G) \geq (1 - \d)\mex_{K_s}(m, K_{r+1}),
	\]
then $G$ has a subgraph~$G'$ with $n'$ vertices and $m' \geq (1-\e)m$ edges such that 
    \[
        \d(G') \geq \beta_r(1-2\e)(m')^{1/2} \qquad\text{and also}\qquad \d(G') \geq \biggl(\frac{r-1}{r}-4\e\biggr)n'.
    \]  
\end{lemma}

\begin{proof}
Let $\e > 0$ be sufficiently small and let
\begin{equation}\label{eq:delta}
\d = \frac{s(s-2)c_{r, s}}{16}\e^2.
\end{equation}

Let $G_0 = G$.  For each $i \geq 0$, if $G_i$ contains a vertex~$v_i$ with $d_{G_i}(v_i) < \beta_r(1-2\e)e(G_i)^{1/2}$, set $G_{i+1} = G_i - v_i$.  Suppose that we delete $\lfloor \e m \rfloor$ edges in this way and that we delete edges incident to $i_0 - 1$ vertices.  (To ensure that we delete exactly $\lfloor \e  m \rfloor$ edges, if necessary we do not delete the final vertex $v_{i_0-1}$, but instead delete the appropriate number of edges incident to it.) Let $G'$ denote the resulting graph.  We have
\[
\k_s(G') \geq \k_s(G) - \sum_{i=0}^{i_0-1} \k_{s-1}\bigl(G_i[N_{G_i}(v_i)]\bigr).
\]
Because $G$ is $K_{r+1}$-free, for each $i$, $G_i[N_{G_i}(v_i)]$ is $K_r$-free.  Hence, the number of copies of $K_s$ in $G_i$ that contain $v_i$ is at most $\ex_{K_{s-1}}\bigl(d_{G_i}(v_i), K_r\bigr)$. By Theorem~\ref{th:Zykov}, for all~$p$, 
\[
    \ex_{K_{s-1}}(p, K_r) = \k_{s-1}(T_{r-1}(p)) \leq \parens[\bigg]{\frac{p}{r-1}}^{s-1}\binom{r-1}{s-1}.  
\]
So, we have
\begin{equation}\label{eq:remainingCliques}
\k_s(G') \geq \k_s(G) - \sum_{i=0}^{i_0-1} \biggl(\dfrac{d_{G_i}(v_i)}{r-1}\biggr)^{s-1}\binom{r-1}{s-1}.
\end{equation}
By assumption, for all $i \leq i_0-1$, $d_{G_i}(v_i) < \beta_r(1-2\e)e(G_i)^{1/2}$.  Moreover, by the definition of~$i_0$,
\begin{equation*}\label{eq:G'DegreeSum}
\sum_{i=0}^{i_0-1} d_{G_i}(v_i) < \e m + d_{G_{i_0 - 1}}(v_{i_0 - 1}) \leq \e m + \beta_r(1-2\e)m^{1/2}.
\end{equation*}
Combining this bound with \eqref{eq:remainingCliques} gives
\begin{align}\label{eq:lostCliquesUB}
\k_s(G) - \k_s(G') &< \biggl\lceil \frac{\e m + \beta_r(1-2\e)m^{1/2}}{\beta_r(1-2\e)(m-\e m)^{1/2}} \biggr\rceil \frac{\binom{r-1}{s-1}}{(r-1)^{s-1}} \beta_r^{s-1}(1-2\e)^{s-1} m^{(s-1)/2} \nonumber\\
&< \e \frac{\binom{r-1}{s-1}}{(r-1)^{s-1}}\bigl(\beta_r(1 - \e)\bigr)^{s-2} m^{s/2}.
\end{align}

On the other hand, if $m$ is sufficiently large, then Theorem~\ref{th:Froh}, Proposition~\ref{prop:colexTuranCliques}, and our assumption on $G$ imply that
\begin{align*}
\k_s(G) - \k_s(G') &\geq \k_s(G) - \mex_{K_s}\bigl((1-\e)m, K_{r+1}\bigr) \\
&\geq (c_{r, s} - \d)m^{s/2} - (1+\e^3)c_{r, s}(1-\e)^{s/2}m^{s/2} \\
&\geq \biggl(c_{r, s} - \d - c_{r, s}\Bigl(1 - \frac{s\e}{2} + \frac{3s(s-2)\e^2}{16}\Bigr)\biggr)m^{s/2}.
\end{align*}
By \eqref{eq:delta} and Proposition~\ref{prop:beta c_s},
\begin{align*}
\k_s(G) - \k_s(G') &\geq \biggl( - \d + \frac{s\e}{2} c_{r, s} - \frac{3s(s-2)\e^2}{16}c_{r, s}\biggr)m^{s/2} \\
&= \biggl(\frac{s\e}{2} c_{r, s} - \frac{s(s-2)\e^2}{4}c_{r, s}\biggr)m^{s/2} \\
&= \e \frac{\binom{r-1}{s-1}}{(r-1)^{s-1}} \beta_r^{s-2}\biggl(1  - \frac{s-2}{2}\e \biggr) m^{s/2} \\
&> \e \frac{\binom{r-1}{s-1}}{(r-1)^{s-1}}\bigl(\beta_r(1 - \e)\bigr)^{s-2} m^{s/2},
\end{align*}
which contradicts~\eqref{eq:lostCliquesUB}.

So, $G$ has a subgraph~$G'$ with $n' \leq n$ vertices and $m' \geq \bigl(1 - \e\bigr)m$ edges such that
\[
\d(G') \geq \beta_r(1-2\e)(m')^{1/2}.
\]
On the other hand, $\d(G') \leq 2m'/n'$, so
\begin{equation}\label{eq:denseSubgraphVertices}
n' \leq \frac{2}{\beta_r(1 - 2\e)}(m')^{1/2}.
\end{equation}
It follows from \eqref{eq:denseSubgraphVertices} and \eqref{eq:beta} that
\begin{equation*}\label{eq:denseSubgraphMinDegree}
\d(G') \geq \beta_r^2(1-2\e)^2\frac{n'}{2} > \beta_r^2(1 - 4\e)\frac{n'}{2} > \biggl(\frac{r-1}{r} - 4\e\biggr) n'.
\end{equation*}
This completes the proof.
\end{proof}

Now we are ready to prove Theorem~\ref{th:mexStability}.  The argument is similar to the proof of the $K_{r+1}$-free case of the Erd\H{o}s--Simonovits stability theorem, Theorem~\ref{th:stability}.

\begin{proof}[Proof of Theorem \ref{th:mexStability}]
Given $\e$, let
\begin{equation}\label{eq:epsilonPrime}
\e' = \frac{\e}{16r+1}.
\end{equation}
Let $G'$ be the graph obtained by inputting $r$, $s$, and $\e'$ into Lemma~\ref{le:denseSubgraph} and let $m' = e(G')$.

By Lemma~\ref{le:denseSubgraph},
\[
e(G') \geq \biggl(\frac{r-1}{r} - 4\e'\biggr)\frac{(n')^2}{2}.
\]
So, if $\e'$ is sufficiently small, Tur\'an's theorem implies that $G'$ contains a copy of~$K_r$ with vertex set~$U = \set{u_1, \dots, u_r}$.  Because $G'$ is $K_{r+1}$-free, every vertex of~$V(G') \setminus U$ has at most~$r-1$ neighbors in $U$.  Let $A = \setof{v \in V(G') \setminus U}{d_U(v) = r - 1}$ and let $B = V(G') \setminus (U \cup A)$.  By definition,
\[
e(U, V(G') \setminus U) \leq (r-1)\abs{A} + (r-2)\abs{B} = (r-1)(n'-r-\abs{B}) + (r-2)\abs{B}.
\]
On the other hand, by Lemma~\ref{le:denseSubgraph},
\[
e(U, V(G') \setminus U) \geq r\biggl(\frac{r-1}{r} - 4\e'\biggr)n' - \binom{r}{2} = (r - 1 - 4\e' r)n' - \binom{r}{2}.
\]
It follows that
\begin{equation}\label{eq:cardinalityOfB}
\abs{B} \leq (r-1)(n'-r) - (r - 1 - 4\e' r)n' + \binom{r}{2} = 4\e' r n' - \binom{r}{2} < 4\e' r n'.
\end{equation}

For $i =1$, \dots,~$r$, let $A_i = \setof{v \in A}{v \nadj u_i}$.  It is easy to see that the $A_i$ partition $A$ and that each $A_i$ is an independent set.  So, if we delete all of the vertices of~$B$ from $G'$, the resulting graph is $r$-partite.

It remains to show that deleting the vertices of~$B$ from $G'$ removes only a small number of edges.  By Lemma~\ref{le:denseSubgraph}, we have $\beta_r(1-2\e)(m')^{1/2} \leq \d(G') \leq 2m'/n'$, which means that
\begin{equation*}
n' \leq \frac{2}{\beta_r(1 - 2\e)}(m')^{1/2}
\end{equation*}
(just as in~\eqref{eq:denseSubgraphVertices}).  This bound, \eqref{eq:cardinalityOfB}, and~\eqref{eq:beta} imply that if $m$ is sufficiently large, then the number of edges of~$G'$ incident to a vertex of~$B$ is at most
\[
\abs{B}(n-\abs{B}) + \binom{\abs{B}}{2} \leq 4\e' r(n')^2 + 8\e'^2 r^2 (n')^2 \leq  \frac{16\e' r}{\beta_r^2(1 - 2\e')^2}m' + \frac{32\e'^2 r^2}{\beta_r^2(1 - 2\e')^2}m' < 16 \e' r m.
\]

It follows from~\eqref{eq:epsilonPrime} that we have deleted at most~$\e m$ edges of~$G$.  This completes the proof.
\end{proof}

\begin{proof}[Proof of Corollary~\ref{cor:mexStabilityGeneral}]
Let $\d'$ and $m_0$ be the values obtained by putting $r$, $s$, and $\e/2$ into Theorem~\ref{th:mexStability}.  Let
\begin{equation}\label{eq:alphaDef}
\alpha = \min\biggl\{\e^2, \frac{\d'}{5\cdot 2^{(s+2)/2}}\biggr\}
\end{equation}
and let
\begin{equation}\label{eq:deltaSmall}
\d = \e\,\alpha.
\end{equation}

We need to show that $G$ has a large subgraph with positive density.  Let $G_0 = G$ and, for each $i \geq 0$, if $G_i$ contains a vertex~$v_i$ with $d_{G_i}(v_i) < \e\, e(G_i)^{1/2}$, set $G_{i+1} = G_i - v_i$.  Suppose that we delete $\lfloor \alpha m \rfloor$ edges in this way and that we delete edges incident to $i_0 - 1$ vertices.  (To ensure that we delete exactly $\lfloor \alpha m \rfloor$ edges, if necessary we do not delete the final vertex~$v_{i_0-1}$, but instead delete the appropriate number of edges incident to it.) Let $G'$ denote the resulting graph.  We have
\begin{equation}\label{eq:stabilityGeneralCliquesLostUB}
\k_s(G) - \k_s(G') \leq \sum_{i=0}^{i_0-1} \k_{s-1}\bigl(G_i[N_{G_i}(v_i)]\bigr) \leq \biggl\lceil \frac{\alpha m + \e m^{1/2}}{\e (m - \alpha m)^{1/2}} \biggr\rceil \binom{\e\, m^{1/2}}{s-1} \leq \frac{2\alpha \e^{s-2}}{(s-1)!} m^{s/2}.
\end{equation}

On the other hand, by Theorem~\ref{th:mexErdosSimonovits} and our assumption on $G$,
\[
\k_s(G) - \k_s(G') \geq \k_s(G) - \mex_{K_s}\bigl((1-\alpha)m, F\bigr) \geq \bigl(c_{r, s} - \d - (1 + \e^3) c_{r, s}(1-\alpha)^{s/2}\bigr)m^{s/2}.
\]
It then follows from~\eqref{eq:deltaSmall} and the fact that $s \geq 3$ that
\[
\k_s(G) - \k_s(G') \geq \biggl(c_{r, s} - \d - c_{r, s}\biggl(1-\frac{s}{4}\alpha\biggr) \biggr)m^{s/2} \geq \frac{s c_{r, s} \alpha}{8} m^{s/2} > \alpha \e^{s-2} m^{s/2},
\]
which contradicts~\eqref{eq:stabilityGeneralCliquesLostUB}.

So, we may assume that $G$ has a subgraph~$G'$ with $m' \geq (1-\alpha)m$ edges and $n'$ vertices such that $\delta(G') \geq \e (m')^{1/2}$.

Let
\begin{equation*}\label{eq:etaDef}
\eta = \e^2 \alpha.
\end{equation*}
Because $G'$ is $F$-free, Theorem~\ref{th:removal} implies that if $m'$ (and hence $n'$) is sufficiently large, then $G'$ can be made $K_{r+1}$-free by removing at most~$ \eta (n')^2$ edges.  Because $\e (m')^{1/2} \leq \d(G') \leq 2m'/n'$, 
\begin{equation}\label{eq:edgesRemovedFromG'}
\eta (n')^2 \leq \frac{4\eta}{\e^2} m' = 4\alpha m'.
\end{equation}

Let $G''$ be the graph obtained by removing edges from $G'$ and let $m'' = e(G'')$.  By~\eqref{eq:edgesRemovedFromG'},
\begin{equation}\label{eq:G''edges}
m'' \geq (1-5\alpha)m.
\end{equation}
By Corollary~\ref{cor:LovaszKKcor}, each of the edges that we have deleted from $G$ was contained in at most $(1+\e^2)\frac{2^{(s-2)/2}}{(s-2)!} m^{(s-2)/2}$ copies of $K_s$ in $G$.  It follows from~\eqref{eq:alphaDef}, \eqref{eq:deltaSmall}, and our assumption on $\k_s(G)$ that
\begin{align*}
\k_s(G'') &\geq \k_s(G) - 5\alpha m \cdot 2 \cdot \frac{2^{(s-2)/2}}{(s-2)!} m^{(s-2)/2} \\
&\geq (1 - \d)\mex_{K_s}(m, F) - \frac{\d'}{2}m^{s/2} \\
&> (1 - \d')\mex_{K_s}(m'', F).
\end{align*}
Hence, by Theorem~\ref{th:mexStability} and our choice of~$\d'$, $G''$ can be made $r$-partite by removing at most~$\frac{\e m''}{2} \leq \frac{\e m}{2}$ edges.  So, by \eqref{eq:alphaDef} and~\eqref{eq:G''edges}, we have removed at most~$5\e^2 m + \frac{\e m}{2} < \e m$ edges of~$G$ in total.  This completes the proof.
\end{proof}



\section{Open questions}\label{sec:open}

The first natural question about extensions of the results we have proved is to consider the dependence of $t$ (the size of the blowup) on the other parameters in Theorem~\ref{th:mexErdosStone}, the analogue of the Erd\H{o}s--Stone theorem. The optimal dependence on $n$ in the latter theorem is $\Omega(\log n)$ (\cite{BollobasErdos73,ChvatalSzemeredi81}). The  following conjecture is the natural analogue.

\begin{conjecture}\label{conj:ErdosStoneSharp}
There exists a constant~$c_{r, s,\e} > 0$ such that it is always possible to take $t \geq c_{r, s,\e} \log m$, and moreover this is best possible up to a constant factor.
\end{conjecture}

We also believe that a supersaturation result should hold for $\mex_{K_s}(m, K_{r+1})$.

\begin{conjecture}\label{conj:supersat}
For all $r \geq s \geq 3$ and all $\e > 0$, there exist $\d > 0$ and $m_0$ such that if $m \geq m_0$ and $G$ is a graph with $e(G)=m$ and
	\[
		\k_s(G) \ge \mex_{K_s}(m,K_{r+1}) + \e m^{s/2},
	\]
then $G$ contains at least~$\d m^{(r+1)/2}$ copies of~$K_{r+1}$.
\end{conjecture}

Finally, when $F$ is a bipartite graph, the Erd\H{o}s--Simonovits theorem, Theorem~\ref{th:ErdosSimonovits}, only says that $\ex(n, F) = o(n^2)$.  In the same way, when $F$ is a graph with $\chi(F) \leq s$, Theorem~\ref{th:mexErdosSimonovits} only tells us that $\mex_{K_s}(m, F) = o(m^{s/2})$.  It would be interesting to determine the order of magnitude of $\mex_{K_s}(n, F)$ in such ``sparse'' cases.  In particular, what are $\mex_{K_s}(m, K_{a, b})$ and $\mex_{K_s}(m, C_{2k})$?

\bibliographystyle{plain}
\bibliography{mex_refs}

\end{document}